\documentclass[11pt,twoside]{amsart}

\usepackage{amsmath}
\usepackage{amsthm}
\usepackage{amsfonts, amssymb}
\usepackage{mathrsfs}
\usepackage[all]{xy}
\usepackage{url}

\usepackage{latexsym}
\usepackage[dvips]{graphicx}

\theoremstyle{plain}
\newtheorem{lema}{Lemma}[section]
\newtheorem{prop}[lema]{Proposition}
\newtheorem{teo}[lema]{Theorem}
\newtheorem{conj}[lema]{Conjecture}

\newtheorem{coro}[lema]{Corollary}
\theoremstyle{remark}

\newtheorem{obs}[lema]{Remark}

\theoremstyle{definition}
\newtheorem{defi}[lema]{Definition}
\newtheorem{ej}[lema]{Example}

\newcommand{\scp}{\mbox{$\ \searrow \! \! \! \! \searrow \! \! \! \! \ \ \ $}}
\newcommand{\sep}{\mbox{$\ \nearrow \! \! \! \! \nearrow \! \! \! \! \ \ \ $}}
\newcommand{\esc}{\mbox{$\ \searrow \! \! \! \! \searrow \! \! \! \! ^e \ $}}

\newcommand{\ceroc}{\mbox{$\ \searrow \! \! \! \! ^0 \ $}}
\newcommand{\unoc}{\mbox{$\ \searrow \! \! \! \! ^1 \ $}}

\newcommand{\enemasunoc}{\mbox{$\ \searrow \! \! \! \! ^{n+1} \ $}}

\newcommand{\se}{\mbox{$\diagup \! \! \! \! \: \! \: \searrow \: $}}

\newcommand{\kp}{\mathcal{K}}
\newcommand{\x}{\mathcal{X}}

\newcommand{\nerve}{\mathcal{N}}

\begin{document}

\title[Strong homotopy types, nerves and collapses]{Strong homotopy types, nerves and collapses}

\author[J.A. Barmak]{Jonathan Ariel Barmak}
\author[E.G. Minian]{Elias Gabriel Minian}
\thanks{The authors' research is partially supported by Conicet and by grant
ANPCyT PICT 17-38280.}

\address{Departamento  de Matem\'atica\\
 FCEyN, Universidad de Buenos Aires\\ Buenos
Aires, Argentina}

\email{jbarmak@dm.uba.ar}
\email{gminian@dm.uba.ar}

\begin{abstract}
We introduce the theory of strong homotopy types of simplicial 
complexes. Similarly to classical simple homotopy theory, the strong 
homotopy types can be described by elementary moves. An elementary move 
in this setting is called a \textit{strong collapse} and it is a 
particular kind of  simplicial collapse. The advantage of using strong 
collapses is the existence and uniqueness of \textit{cores} and their 
relationship with the \textit{nerves} of the complexes.  From this 
theory we derive new results for studying simplicial collapsibility with 
a different point of view. We analyze vertex-transitive simplicial $G$-actions 
and prove a particular case of the Evasiveness conjecture for simplicial 
complexes. Moreover, we reduce the general conjecture to the class of 
\textit{minimal complexes}. We also strengthen a result of V. Welker on 
the barycentric subdivision of collapsible complexes. We obtain this and 
other results on collapsibility of polyhedra by means of the 
characterization of the different notions of collapses in terms of 
finite topological spaces.
\end{abstract}




\maketitle

\section{Introduction}
The notion of simplicial collapse, introduced by J.H.C. Whitehead in the late thirties \cite{Whi}, is a fundamental tool in algebraic topology and in combinatorial geometry \cite{Coh,Gla}. Collapsible polyhedra are the centre of various famous problems, perhaps the most relevant of them being the Zeeman Conjecture \cite{Zee}. Although an elementary simplicial collapse is very simple to describe, it is not easy to determine whether a simplicial complex is collapsible. In fact, the Zeeman Conjecture, which states that $K\times I$ is collapsible for any contractible polyhedron $K$ of dimension $2$, is still an open problem. One of the main difficulties when studying collapsibility, or simple homotopy theory in general, is the non-uniqueness of what we call \it cores\rm. A \textit{core} of a complex $K$, in this setting, is a minimal element in the poset of all subcomplexes which expand to $K$. For instance, there exist collapsible complexes which collapse to nontrivial subcomplexes with no free faces.

In this article we introduce and develope the theory of strong homotopy types of simplicial complexes, which has interesting applications to problems of collapsibility. Similarly to the case of simple homotopy types, the strong homotopy types can be described in terms of elementary moves. A vertex $v$ of a simplicial complex $K$ is called \it dominated \rm if its link is a simplicial cone. An elementary strong collapse  $K \esc K\smallsetminus v$ consists of removing from a simplicial complex $K$ the open star of a dominated vertex $v$.  Surprisingly the notion of strong homotopy equivalence also arises from the concept of contiguity classes. It is not hard to see that a strong collapse is a particular case of a classical collapse. This new kind of collapse is much easier to handle for various reasons: the existence and uniqueness of cores, the fact that strong homotopy equivalences between \it minimal complexes \rm are isomorphisms and the relationship between this theory and  the \it nerves \rm of the complexes. A finite simplicial complex $K$ is minimal if it has no dominated vertices and a core of a finite simplicial complex $K$ in this context is a minimal subcomplex $K_0\subseteq K$ such that $K\scp K_0$. Minimal complexes appeared already in the literature under the name of \it taut \rm complexes \cite{Gru}. B. Gr\"unbaum investigated  the relationship between minimal complexes and nerves, but apparently he was unaware of the significant relationship between these notions and collapsibility. We will show that, in contrast to the classical situation, in this context all the cores of a complex $K$ are isomorphic and one can reach the core of a complex by iterating the nerve operation. This provides a new method to study collapsibility from a different viewpoint.

In order to investigate the difference between simplicial collapses and strong collapses, in Section \ref{eva} we relax the notion of strong collapse and define inductively divers notions of 
collapses which lie between the two concepts. In this way, the notion of non-evasiveness appears naturally in our context. The concept of non-evasiveness, commonly used in combinatorics and combinatorial geometry, is motivated by problems in graph theory \cite{Bjo3,Kah,Lut1,Lut2,Wel}. A non-evasive simplicial complex is collapsible (in the classical sense) but the converse is not true. Our approach is slightly different from the previous treatments of this subject, since we are more interested in understanding the difference between the various notions of collapses from a geometric point of view. One of the most significant questions related to this concept is the so-called \it Evasiveness conjecture for simplicial complexes \rm which asserts that a non-evasive simplicial complex endowed with a vertex-transitive action of some group $G$ is a simplex. The conjecture is still open for non-evasive complexes but it is known to be false for collapsible complexes in general \cite{Lut1,Lut2}.

A simplicial complex $K$ is vertex-homogeneous if for any pair of vertices $v,w\in K$, there exists an automorphism $\varphi\in Aut(K)$ such that $\varphi(v)=w$. We will show that a strong collapsible complex which is vertex-homogeneous is a simplex. This proves a particular case of the Evasiveness conjecture mentioned above and sheds some light on the general situation. Moreover, we will see that the square nerve $\nerve^2(K)$ of any vertex-homogeneous complex $K$ is its core and that any vertex-homogeneous complex is isomorphic to an $n$-th multiple of its core. As an easy and direct consequence of these results, one obtains that any vertex-homogeneous complex with a prime number of vertices is minimal or a simplex. One of the main results that we prove in this direction asserts that the core of a vertex-homogeneous and non-evasive complex is also
 vertex-homogeneous and non-evasive. In view of this result, one derives that the study of the Evasiveness conjecture for simplicial complexes can be reduced to the class of minimal (and non-evasive) complexes.

The theory of strong homotopy types for polyhedra is motivated by the homotopy theory of finite topological spaces. One can associate to any finite simplicial complex $K$, a finite $T_0$-space $\x(K)$ which corresponds to the poset of simplices of $K$. Conversely, one can associate to a given finite $T_0$-space $X$ the simplicial complex $\kp(X)$ of its non-empty chains.  
In \cite{Bar2} we have introduced the notion of collapse in the setting of finite spaces and proved that a collapse $X\searrow  Y$ between finite spaces induces a collapse $\kp(X) \searrow \kp(Y)$ between their associated simplicial complexes and a simplicial collapse $K \searrow L$ induces a collapse between the associated finite spaces. One advantage of working with finite spaces is that the elementary collapses in this context are very simple to handle and describe: they consist of removing a single point of the space, which is called a \it weak point. \rm In this paper we extend these results to the non-evasive case and the strong homotopy case. We define the notion of non-evasive collapse in the setting of finite spaces and prove that this notion corresponds exactly to its simplicial version by means of the functors $\x$ and $\kp$. Analogously we prove that the concept of strong homotopy equivalence in the simplicial context corresponds to the notion of homotopy equivalence in the setting of finite spaces. We use these results on finite spaces to derive results on polyhedra: we prove that a simplicial complex $K$ is strong collapsible if and only if its barycentric subdivsion $K'$ is strong collapsible and that the join $KL$ of two simplicial complexes is strong collapsible if and only if $K$ or $L$ is strong collapsible. As another direct consequence of our results one can easily deduce a stronger version of a known result of V. Welker on the barycentric subdivision of collapsible complexes \cite{Wel}. Indeed, we show that a simplicial collapse $K\searrow L$ determines a $1$-collapse $K'\unoc L'$. In particular if $K$ is collapsible, $K'$ is $1$-collapsible and therefore non-evasive. The converse of this result is still an open problem but it holds in the case that $K'$ is strong collapsible.

\section{Strong homotopy types} \label{sectionsht}  

We recall the notion of contiguity classes of simplicial maps. A standard reference for this is \cite{Spa}. Two simplicial maps $\varphi, \psi :K\to L$ are said to be \textit{contiguous} if $\varphi (\sigma)\cup \psi (\sigma)$ is a simplex of $L$ for every $\sigma \in K$. The equivalence classes of the equivalence relation generated by contiguity are called \textit{contiguity classes}. Contiguous maps are homotopic at the level of geometric realizations, but this notion is strictly stronger than usual homotopy. The difference between \textit{strong homotopy types} and the usual notion of homotopy types lies in this distiction between contiguity classes and homotopic maps.

Given a simplicial complex $K$ and a simplex $\sigma\in K$, the (closed) star of $\sigma$ will be denoted by $st_K(\sigma)$. It is the subcomplex of simplices $\tau$ such that $\tau \cup \sigma$ is a simplex of $K$. The \textit{link} $lk_K(\sigma)$ is the subcomplex of $st(\sigma)$ of simplices disjoint from $\sigma$.
If $K$ and $L$ are two disjoint complexes, the \textit{join} $K*L$ (or $KL$) is the complex whose simplices are those of $K$, those of $L$ and unions of simplices of $K$ and $L$.
A \textit{simplicial cone} is the join $aK$ of a complex $K$ and vertex $a$ not in $K$. 

All the simplicial complexes we deal with are assumed to be finite. We will write  ``complex" or ``simplicial complex" instead of ``finite simplicial complex".

\begin{defi}
Let $K$ be a complex and let $v\in K$ be a vertex. We denote by $K\smallsetminus v$ the full subcomplex of $K$ spanned by the vertices different from $v$ (the \textit{deletion} of the vertex $v$). We say that there is an \textit{elementary strong collapse} from $K$ to $K\smallsetminus v$ if $lk_K(v)$ is a simplicial cone $v'L$. In this case we say that $v$ is \textit{dominated} (by $v'$) and we denote $K \esc K\smallsetminus v$. There is a \textit{strong collapse} from a complex $K$ to a subcomplex $L$ if there exists a sequence of elementary strong collapses that starts in $K$ and ends in $L$. In this case we write $K \scp L$. The inverse of a strong collapse is a \textit{strong expansion} and two complexes $K$ and $L$ have the same \textit{strong homotopy type} if there is a sequence of strong collapses and strong expansions that starts in $K$ and ends in $L$. 
\end{defi}

\begin{figure}[h]
\begin{center}
\includegraphics[scale=0.5]{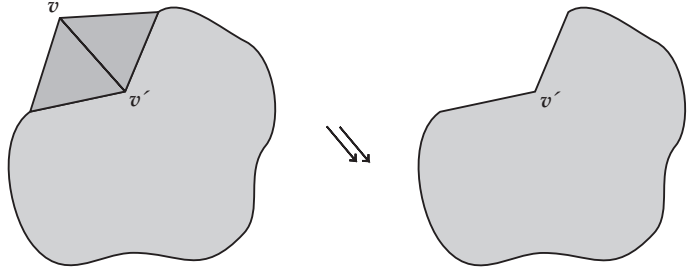}

\caption{An elementary strong collapse. The vertex $v$ is dominated by $v'$.}
\end{center}
\end{figure}

\begin{obs}
A vertex $v$ is dominated by a vertex $v'\neq v$ if and only if every maximal simplex that contains $v$ also contains $v'$.
\end{obs}

\begin{obs} \label{isosht}
It is not hard to see that isomorphic complexes have the same strong homotopy type. If $K$ is a complex and $v\in K$ is a vertex, we can consider for a vertex $v'$ not in $K$ the complex $L=K+ (v'st_K(v))=(K\smallsetminus v)+ (v'vlk_K(v))$. Since $lk_L(v')=vlk_K(v)$, then $L\scp K$. Moreover, by symmetry $L\scp L\smallsetminus v=\widetilde{K}$. One can use this process repeatedly to show that isomorphic complexes are strong homotopy equivalent.  
\end{obs}

The notion of strong collapse is related to the usual notion of collapse. Recall that if $K$ is a complex and $\sigma$ is a free face of $\tau$ (i.e. $\tau$ is the unique simplex of $K$ having $\sigma$ as a proper face) then we say that there is an \textit{elementary collapse} from $K$ to $K\smallsetminus \{\sigma, \tau\}$. A \textit{collapse} $K\searrow L$ from a complex $K$ to a subcomplex $L$ (or an \textit{expansion} $L\nearrow K$) is a sequence of elementary collapses from $K$ to $L$. Two complexes are said to have the same \textit{simple homotopy type} if there is a sequence of collapses and expansions from one to the other. If $K$ and $L$ have the same simple homotopy type, we write $K\se L$.

It is not hard to prove that if $K$ and $L$ are subcomplexes of a same complex, then $K+ L \searrow L$ if and only if $K\searrow K\cap L$. A complex is \textit{collapsible} if it collapses to a point. For instance, simplicial cones are collapsible. It is easy to prove that a complex $K$ is collapsible if and only if the simplicial cone $aK$ collapses to its base $K$.

\begin{obs} \label{weaker}
If $v\in K$ is dominated, $lk(v)$ is collapsible and therefore $st(v)=v lk(v)\searrow lk(v)=st(v) \cap (K\smallsetminus v)$. Then $K=st(v) + (K\smallsetminus v)\searrow K\smallsetminus v$. Thus, the usual notion of collapse is weaker than the notion of strong collapse.
\end{obs}

If two simplicial maps $\varphi , \psi :K\to L$ lie in the same contiguity class, we will write $\varphi \sim \psi$. It is easy to see that if $\varphi _1, \varphi _2 :K\to L$ and $\psi _1 ,\psi _2 :L \to M$ are simplicial maps such that $\varphi _1 \sim \varphi _2$ and $\psi _1 \sim \psi _2$, then $\psi _1 \varphi _1 \sim \psi _2 \varphi _2$.

\begin{defi} \label{stronghomotopyequivalence}
A simplicial map $\varphi :K\to L$ is a \textit{strong equivalence} if there exists $\psi :L \to K$ such that $\psi \varphi \sim 1_K$ and $\varphi \psi \sim 1_L$. If there is a strong equivalence $\varphi :K \to L$ we write $K\sim L$.
\end{defi}

The relation $\sim$ is clearly an equivalence relation.

\begin{defi}
A complex $K$ is a \textit{minimal complex} if it has no dominated vertices.
\end{defi}

\begin{prop}
Let $K$ be a minimal complex and let $\varphi :K \to K$ be simplicial map which lies in the same contiguity class as the identity. Then $\varphi$ is the identity.
\end{prop}
\begin{proof}
We may assume that $\varphi$ is contiguous to $1_K$. Let $v\in K$ and let $\sigma \in K$ be a maximal simplex such that $v\in \sigma$. Then $\varphi (\sigma)\cup \sigma$ is a simplex, and by the maximality of $\sigma$, $\varphi(v)\in \varphi (\sigma)\cup \sigma = \sigma$. Therefore every maximal simplex which contains $v$, also contains $\varphi (v)$. Hence, $\varphi (v)=v$, since $K$ is minimal.
\end{proof}

\begin{coro} \label{seminimal}
A strong equivalence between minimal complexes is an isomorphism.
\end{coro}

\begin{prop} \label{shts}
Let $K$ be a complex and $v\in K$ a vertex dominated by $v'$. Then, the inclusion $i: K \smallsetminus v \hookrightarrow K$ is a strong equivalence. In particular, if two complexes $K$ and $L$ have the same strong homotopy type, then $K \sim L$.
\end{prop}
\begin{proof}
Define a vertex map $r: K \to K\smallsetminus v$ which is the identity on $K\smallsetminus v$ and such that $r(v)=v'$. If $\sigma \in K$ is a simplex with $v\in \sigma$, consider $\sigma'\supseteq \sigma$ a maximal simplex. Therefore $v'\in \sigma'$ and $r (\sigma)=\sigma \cup \{ v' \} \smallsetminus \{ v \} \subseteq \sigma'$ is a simplex of $K\smallsetminus v$. Moreover $ir (\sigma) \cup \sigma = \sigma \cup \{ v' \}\subseteq \sigma'$ is a simplex of $K$. This proves that $r$ is simplicial and that $ir$ is contiguous to $1_K$. Therefore, $i$ is a strong equivalence.
\end{proof}

\begin{defi}
A \textit{core} of a complex $K$ is a minimal subcomplex $K_0\subseteq K$ such that $K\scp K_0$.
\end{defi}

\begin{teo}
Every complex has a core and it is unique up to isomorphism. Two complexes have the same strong homotopy type if and only if their cores are isomorphic.
\end{teo}

\begin{proof}
A core of a complex can be obtained by removing dominated points one at the time. If $K_1$ and $K_2$ are two cores of $K$, they have the same strong homotopy type and by Proposition \ref{shts}, $K_1 \sim K_2$. Since they are minimal, by Corollary \ref{seminimal} they are isomorphic.

Let $K$, $L$ be two complexes. If they have the same strong homotopy type, then also their cores $K_0$ and $L_0$ do. As above, we conclude that $K_0$ and $L_0$ are isomorphic.

Conversely, if $K_0$ and $L_0$ are isomorphic, they have the same strong homotopy type by Remark \ref{isosht}. Therefore $K$ and $L$ have the same strong homotopy type. 
\end{proof}

If $K$ and $L$ are two complexes such that $K\sim L$ and $K_0 \subseteq K$, $L_0 \subseteq L$ are their cores, then $K_0 \sim L_0$ and therefore $K_0$ and $L_0$ are isomorphic. Hence, we deduce the following

\begin{coro} \label{equivstrong}
Two complexes $K$ and $L$ have the same strong homotopy type if and only if $K\sim L$.
\end{coro}

\begin{ej} \label{nosc}
The homogeneous 2-complex of Figure \ref{noscp} is collapsible, moreover it is non-evasive (see Section \ref{eva}). However, it is a minimal complex and therefore it does not have the strong homotopy type of a point. 

\begin{figure}[!h]
\begin{center} 
\includegraphics[scale=0.7]{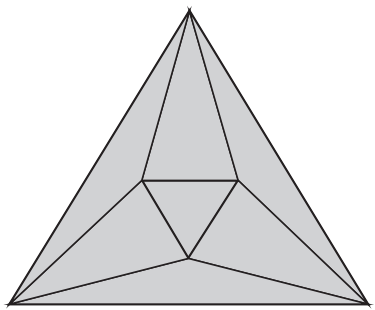} 
\caption{A collapsible complex which is not strong collapsible.}\label{noscp}
\end{center}
\end{figure}
\end{ej}

A complex is said to be \textit{strong collapsible} if it strong collapses to one of its vertices or, equivalently, if it has the strong homotopy type of a point. 

The barycentric subdivision $K'$ of a complex $K$ has as vertices the barycenters $\hat{\sigma}$ of the simplices of $K$ and the simplices $\{\hat{\sigma}_0, \hat{\sigma}_1, \ldots, \hat{\sigma}_r \}$ of $K'$ correspond to chains $\sigma_0\subseteq \sigma_1 \subseteq \ldots \subseteq \sigma_r$. It is well known that barycentric subdivisions preserve the simple homotopy type. In fact, stellar subdivisions do. In contrast to the case of simple homotopy, a complex and its barycentric subdivision need not have the same strong homotopy type. For example, the boundary of a 2-simplex and its barycentric subdivision are minimal non-isomorphic complexes, and therefore they do not have the same strong homotopy type. However, we will show that a complex $K$ is strong collapsible if and only if its barycentric subdivision is strong collapsible. Note that minimality is not preserved by barycentric subdivisions. In fact $K'$ is minimal if and only if $K$ has no free faces, or equivalently if $K'$ has no free faces. If $\sigma$ is a free face of $\tau$ in $K$, then $\hat{\sigma}$ is dominated by $\hat{\tau}$ in $K'$.

The following result relates the notion of strong equivalence with the concept of simple homotopy equivalence. For definitions and results on simple homotopy equivalences, the reader is referred to \cite{Coh}.  


\begin{prop}
Strong equivalences are simple homotopy equivalences.
\end{prop}
\begin{proof}
Let $\varphi :K\to L$ be a strong equivalence. Let $K_0$ be a core of $K$ and $L_0$ a core of $L$. Then, the inclusion $i: K_0\hookrightarrow K$ is a strong equivalence and there exists a strong equivalence $r:L \to L_0$ which is a homotopy inverse of the inclusion $L_0\hookrightarrow L$. Since $K_0$ and $L_0$ are minimal complexes, the strong equivalence $r\varphi i$ is an isomorphism and in particular $|r\varphi i|$ is a simple homotopy equivalences. By Remark \ref{weaker}, $|i|$ and $|r|$ are also simple homotopy equivalences, and then so is $|\varphi |$.  
\end{proof}

It is not known whether $K*L$ is collapsible only if one of $K$ or $L$ is (see \cite{Wel}), but the analogous result is true for strong collapsibility.

\begin{prop}
Let $K$ and $L$ be two complexes. Then, $K*L$ is strong collapsible if and only if $K$ or $L$ is strong collapsible.
\end{prop}
\begin{proof}
Suppose $v$ is a dominated vertex of $K$. Then $lk _K (v)$ is a cone and therefore $lk _{K*L} (v)=lk _K (v)*L$ is a cone. Hence, $v$ is also dominated in $K*L$. Thus, if $K$ strong collapses to a vertex $v_0$, $K*L \scp v_0L\scp v_0$.

Conversely, assume $K*L$ is strong collapsible. Let $v\in K*L$ be a dominated point and suppose without loss of generality that $v\in K$. Then $lk _{K*L} (v)=lk _{K} (v)*L$ is a cone. Therefore $lk _{K} (v)$ is a cone or $L$ is a cone. If $L$ is a cone, it is strong collapsible and we are done. Suppose then that $lk _K (v)$ is a cone. Since $(K\smallsetminus v)*L=(K*L)\smallsetminus v$ is strong collapsible, by induction $K\smallsetminus v$ or $L$ is strong collapsible and since $K\scp K \smallsetminus v$, $K$ or $L$ is strong collapsible. 
\end{proof}

\section{The nerve of a complex}

We introduce an application which transforms a complex in another complex with the same homotopy type. This construction was previously studied by B. Gr\"unbaum in \cite{Gru} (see also \cite{Lut1}). We arrived to this concept independently when studying $\check{\mathrm{C}}$ech cohomology of finite topological spaces (see \cite{Bar6}). This construction was already known to preserve the homotopy type of a complex, but we will show that when we apply the construction twice, one obtains a complex with the same strong homotopy type. Moreover, the nerve  can be used to obtain the core of a complex.

\begin{defi}
Let $K$ be a complex. The \textit{nerve} $\nerve (K)$ of $K$ is a complex whose vertices are the maximal simplices of $K$ and whose simplices are the sets of maximal simplices of $K$ with nonempty intersection. Given $n\ge 2$, we define recursively $\nerve^n(K)=\nerve(\nerve^{n-1}(K))$.
\end{defi}

A proof of the fact that $K$ is homotopy equivalent to its nerve can be given invoking a Theorem of Dowker \cite{Dow}. Let $V$ be the set of vertices of $K$ and $S$ the set of its maximal simplices. Define the relation $R\subseteq V\times S$ by $v R \sigma $ if $v\in \sigma$. Following Dowker, one can consider two complexes. The simplices of the first one are the finite subsets of $V$ which are related with a same element of $S$. This complex coincides with the original complex $K$. The second complex has as simplices the finite subsets of $S$ which are related with a same element of $V$. This complex is $\nerve (K)$. The Theorem of Dowker concludes that $|K|$ and $|\nerve (K)|$ have the same homotopy type.

\begin{ej}
Let $K$ be the complex on the left in Figure \ref{nerviop} with five maximal simplices $\sigma _1, \sigma _2, \sigma _3, \sigma _4$ and $\sigma _5$. Its nerve $\nerve (K)$ is the complex on the right whose vertices are the maximal simplices of $K$.
\bigskip
\bigskip

\begin{figure}[h]
\begin{minipage}{8cm}
\begin{center} \includegraphics[scale=0.8]{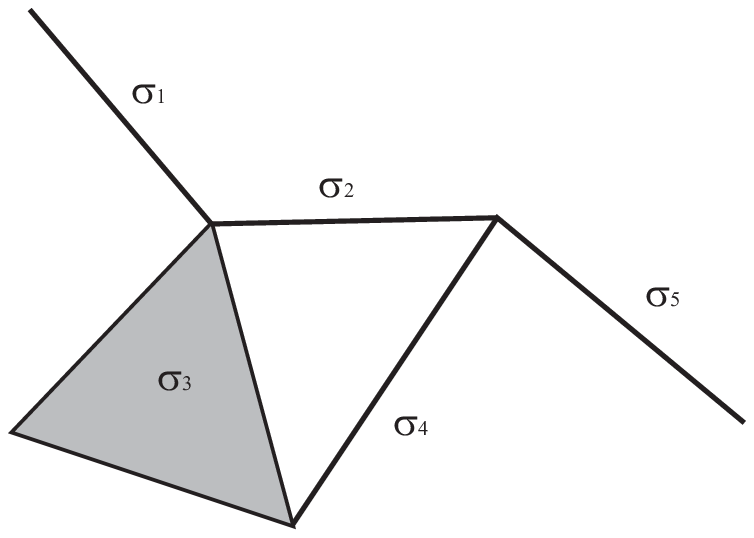} \end{center}
\end{minipage}
\
\begin{minipage}{6cm}
\begin{center} \includegraphics[scale=0.8]{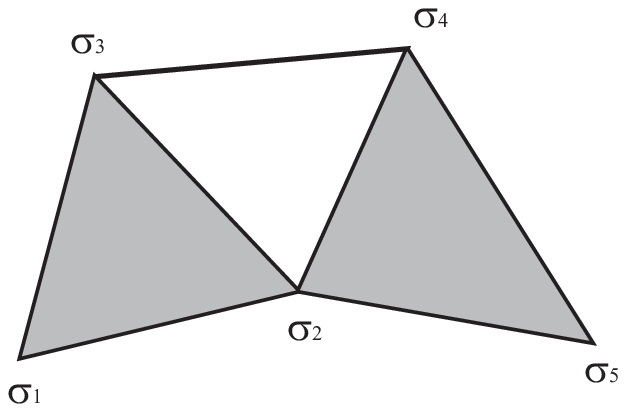} \end{center}
\end{minipage}
\bigskip
\begin{center} \caption{A complex on the left and its nerve on the right.}\label{nerviop} \end{center}
\end{figure}

\end{ej}

If $\nerve^r(K)=*$ for some $r\ge 1$, then $|K|$ is contractible. But there are contractible complexes such that $\nerve^r(K)$ is not a point for any $r$. For instance, if $K$ is the complex of Example \ref{nosc}, $\nerve(K)$ has more vertices than $K$, but $\nerve^2(K)$ is isomorphic to $K$. Therefore $\nerve^r(K)\neq *$ for every $r$.

We will see that, in fact, there is a strong collapse from $K$ to a complex isomorphic to $\nerve^2(K)$ and that there exists $r$ such that $\nerve^r(K)=*$ if and only if $K$ is strong collapsible.

\begin{lema} \label{alema1}
Let $L$ be a full subcomplex of a complex $K$ such that every vertex of $K$ which is not in $L$ is dominated by some vertex in $L$. Then $K\scp L$.
\end{lema}
\begin{proof}
Let $v$ be a vertex of $K$ which is not in $L$. By hypothesis, $v$ is dominated and then $K\scp K\smallsetminus v$. Now suppose $w$ is a vertex of $K\smallsetminus v$ which is not in $L$. Then, the link $lk_K(w)$ in $K$ is a simplicial cone $aM$ with $a\in L$. Therefore, the link $lk_{K\smallsetminus v}(w)$ in $K\smallsetminus v$ is $a(M\smallsetminus v)$. By induction $K\smallsetminus v \scp L$ and then $K\scp L$. 
\end{proof}

\begin{prop} \label{aprop11}
Let $K$ be a simplicial complex. Then, there exists a complex $L$ isomorphic to $\nerve^2(K)$ such that $K\scp L$.
\end{prop}
\begin{proof}
A vertex of $\nerve^2(K)$ is a maximal family $\Sigma =\{\sigma _0, \sigma _1, \ldots , \sigma _r \}$ of maximal simplices of $K$ with nonempty intersection. Consider a vertex map $\varphi : \nerve^2(K)\to K$ such that $\varphi(\Sigma) \in \bigcap\limits_{i=0}^r \sigma _i$. This is a simplicial map for if $\Sigma _0,\Sigma _1,\ldots , \Sigma _r$ constitute a simplex of $\nerve^2(K)$, then there is a common element $\sigma$ in all of them, which is a maximal simplex of $K$. Therefore $\varphi (\Sigma _i)\in \sigma $ for every $0\le i\le r$ and then $\{\varphi (\Sigma _1), \varphi (\Sigma _2), \ldots , \varphi (\Sigma _r) \}$ is a simplex of $K$.

The vertex map $\varphi $ is injective. If $\varphi (\Sigma _1)=v=\varphi (\Sigma _2)$ for $\Sigma _1=\{\sigma _0, \sigma _1, \ldots , \sigma _r \}$, $\Sigma _2=\{\tau _0, \tau _1, \ldots , \tau _t \}$, then $v\in \sigma _i$ for every $0\le i\le r$ and $v\in \tau _i$ for every $0\le i\le t$. Therefore $\Sigma _1 \cup \Sigma _2$ is a family of maximal simplices of $K$ with nonempty intersection. By the maximality of $\Sigma _1$ and $\Sigma _2$, $\Sigma _1=\Sigma _1 \cup \Sigma _2=\Sigma _2$. 

Suppose $\Sigma _0, \Sigma_1 ,\ldots , \Sigma _r$ are vertices of $\nerve^2(K)$ such that $v_0=\varphi (\Sigma _0), v_1=\varphi (\Sigma _1), \ldots ,$ $v_r=\varphi (\Sigma _r)$ constitute a simplex of $K$. Let $\sigma $ by a maximal simplex of $K$ which contains $v_0,v_1,\ldots ,v_r$. Then, by the maximality of the families $\Sigma _i$, $\sigma \in \Sigma _i$ for every $i$ and therefore $\{ \Sigma _0, \Sigma_1 ,\ldots , \Sigma _r \}$ is a simplex of $\nerve^2(K)$.

This proves that $L=\varphi (\nerve^2(K))$ is a full subcomplex of $K$ which is isomorphic to $\nerve^2(K)$.

Now, suppose $v$ is a vertex of $K$ which is not in $L$. Let $\Sigma $ be the set of maximal simplices of $K$ which contain $v$. The intersection of the elements of $\Sigma$ is nonempty, but $\Sigma $ could be not maximal. Let $\Sigma '\supseteq \Sigma$ be a maximal family of maximal simplices of $K$ with nonempty intersection. Then $v'=\varphi (\Sigma ')\in L$ and if $\sigma $ is a maximal simplex of $K$ which contains $v$, then $\sigma \in \Sigma \subseteq \Sigma '$. Hence, $v' \in \sigma$. Therefore $v$ is dominated by $v'$. By Lemma \ref{alema1}, $K\scp L$.   
\end{proof}

\begin{ej}
In the complex $K$ of Figure \ref{nervecuadrado} there are, according to the previous Proposition, two possible ways of regarding $\nerve ^2 (K)$ as a subcomplex of $K$. One of these is the full subcomplex with vertices $a,b,c$ and $d$. 

\begin{figure}[h]
\begin{center}
\includegraphics[scale=0.8]{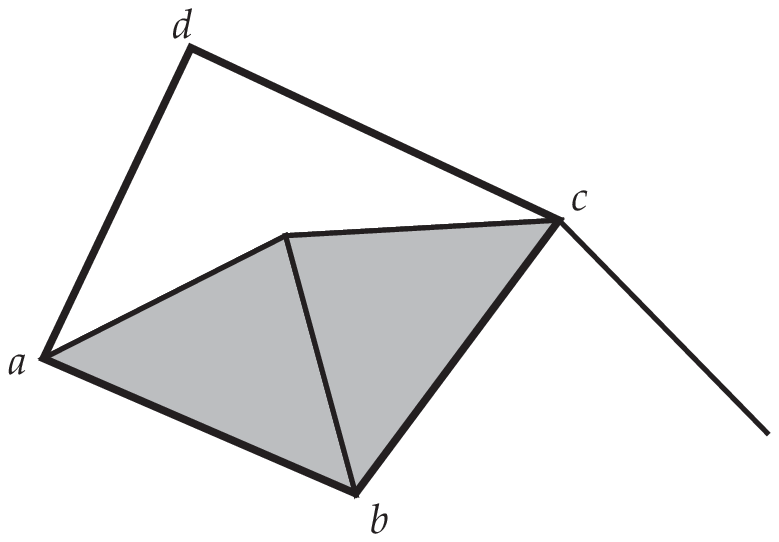}
\caption{$\nerve ^2(K)$ as a subcomplex of $K$.}\label{nervecuadrado}
\end{center}
\end{figure}
\end{ej}

The nerve can be used to characterize minimal complexes as the following result shows (cf. \cite{Gru} Theorem 9).

\begin{lema}
A complex $K$ is minimal if and only if $\nerve^2(K)$ is isomorphic to $K$.
\end{lema} 
\begin{proof}
By Proposition \ref{aprop11}, there exists a complex $L$ isomorphic to $\nerve^2(K)$ such that $K\scp L$. Therefore, if $K$ is minimal, $L=K$.

If $K$ is not minimal, there exists a vertex $v$ dominated by other vertex $v'$. If $v$ is contained in each element of a maximal family $\Sigma$ of maximal simplices of $K$ with nonempty intersection, then the same occurs with $v'$. Therefore, we can define the map $\varphi $ of the proof of Proposition \ref{aprop11} so that $v$ is not in its image. Therefore, $L=\varphi (\nerve^2(K))$ is isomorphic to $\nerve^2(K)$ and has less vertices than $K$. Thus, $\nerve^2(K)$ and $K$ cannot be isomorphic.  
\end{proof}

The sequence $K, \nerve^{2}(K), \nerve^{4}(K), \nerve^{6}(K), \ldots $ is a decreasing sequence of subcomplexes of $K$ (up to isomorphism). Therefore, there exists $n\ge 1$ such that $\nerve^{2n}(K)$ and $\nerve^{2n+2}(K)$ are isomorphic. Then $K$ strongly collapses to a subcomplex $L$ which is isomorphic to $\nerve^{2n}(K)$ and which is minimal. Thus, we have proved the following 

\begin{prop}
Given a complex $K$, there exists $n\ge 1$ such that $\nerve^n(K)$ is isomorphic to the core of $K$.
\end{prop}

\begin{teo}
Let $K$ be a complex. Then, $K$ is strong collapsible if and only if there exists $n\ge 1$ such that $\nerve^n(K)$ is a point.
\end{teo}
\begin{proof}
If $K$ is strong collapsible, its core is a point and then, there exists $n$ such that $\nerve^n(K)=*$ by the previous proposition. If $\nerve^n(K)=*$ for some $n$, then $\nerve^{n+1}(K)$ is also a point. Therefore there exists an even positive integer $r$ such that $\nerve^r(K)=*$, and $K\scp *$ by Proposition \ref{aprop11}.
\end{proof}

%

\section{Finite topological spaces}

The theory of strong homotopy types was motivated by the homotopy theory of finite spaces developed by R.E. Stong in \cite{Sto}. In this section we recall some basic concepts and results on finite spaces and we establish a correspondence between strong homotopy types of finite simplicial complexes and homotopy types of finite spaces. For a comprehensive treatment of this subject we refer the reader to \cite{Bar6, Bar2, Bar3, May, May2, Mcc, Sto}.

In general, a complex $K$ and its barycentric subdivision $K'$ do not have the same strong homotopy type. We will use the relationship between strong homotopy types of complexes and finite spaces to prove that $K$ is strong collapsible if and only if $K'$ is strong collapsible. Finally, we prove an analogous result for finite spaces.

\bigskip

Given a finite topological space $X$, for each $x\in X$ we denote by $U_x$ the minimal open set containing $x$. The relation $x\le y$ if $x\in U_y$ defines a preorder on the set $X$, that is to say a reflexive and transitive relation. If $X$ is in addition a $T_0$-space (i.e. for every pair of points there exists an open set which contains one and only one of them), $\le$ is antisymmetric. This establishes a correspondence between finite $T_0$-spaces and finite posets. Moreover, continuous maps correspond to order-preserving maps.

\bigskip
\bigskip

\begin{figure}[h]
\qquad \qquad
\begin{minipage}{7cm}
\begin{center}
\includegraphics[scale=0.7]{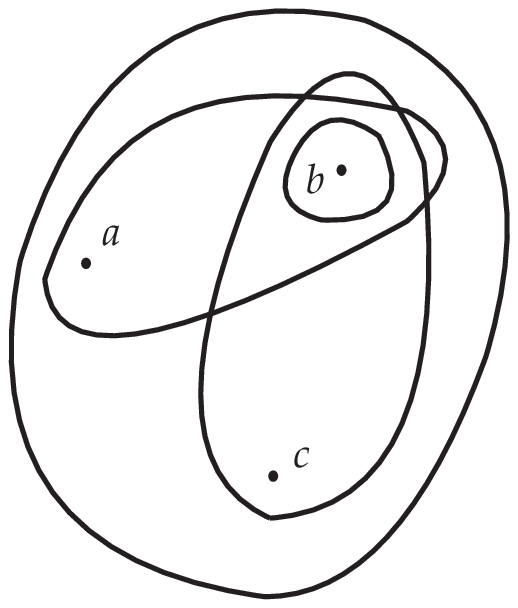}
\end{center}
\end{minipage}
\begin{minipage}{5cm}
\includegraphics[scale=0.8]{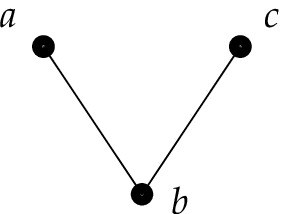}
\end{minipage}
\begin{center} 
\caption{A finite $T_0$-space of three points and the}

Hasse diagram of the corresponding poset. \end{center}
\end{figure}


Given two finite $T_0$-spaces $X$, $Y$, we denote by $Y^X$ the space of maps with the compact-open topology. This finite $T_0$-space corresponds to the set of order-preserving maps $X\to Y$ with the pointwise ordering, i.e. $f\le g$ if $f(x)\le g(x)$ for every $x\in X$. The following result characterizes homotopic maps between finite spaces in terms of posets.

\begin{prop} \label{homotopias}
Let $f,g: X\to Y$ be two maps between finite spaces. Then $f\simeq g$ if and only if there is a fence $f=f_0\le f_1\ge f_2\le \ldots f_n=g$. Moreover, if $A\subseteq X$, then $f\simeq g \ rel \ A$ if and only if there exists a fence $f=f_0\le f_1\ge f_2\le \ldots f_n=g$ such that $f_i|_{A}=f|_{A}$ for every $0\le i\le n$.  
\end{prop}

If $x$ is a point in a finite $T_0$-space $X$, $F_x$ denotes the set of points which are greater than or equal to $x$ or, in other words, the closure of $\{x\}$. $\hat{U}_x=U_x\smallsetminus \{x\}$ is the set of points which are strictly smaller than $x$ and $\hat{F}_x=F_x\smallsetminus \{x\}$. The \textit{link} $\hat{C}_x$ of a point $x$ in a finite $T_0$-space $X$ is the subspace $\hat{U}_x\cup \hat{F}_x$ of points comparable with $x$ and distinct from $x$. Finally, we denote $C_x=U_x\cup F_x$ the set of points comparable with $x$.

\begin{defi}
A point $x\in X$ is a \textit{beat point} if $\hat{U}_x$ has a maximum or if $\hat{F}_x$ has a minimum. In the first case, $x$ is called a \textit{down beat point} and in the second, an \textit{up beat point}. In other words, $x$ is a beat point if it \textit{covers} exactly one element or if it is covered by exactly one element. 
\end{defi}

\begin{defi}
If $x$ is a beat point of a finite $T_0$-space $X$, we say that there is an \textit{elementary strong collapse} from $X$ to $X\smallsetminus \{x\}$. There is a \textit{strong collapse} from $X$ to a subspace $Y$ (or a \textit{strong expansion} from $Y$ to $X$) if there exists a sequence of elementary strong collapses starting in $X$ and finishing in $Y$. We denote this situation by $X\scp Y$ or $Y\sep X$.
\end{defi}

It is not hard to prove that if $x\in X$ is a beat point, then $X\smallsetminus \{x\}$ is a strong deformation retract of $X$. Therefore, strong expansions of finite spaces are strong deformation retracts.

A finite $T_0$-space is said to be a \textit{minimal finite space} if it has no beat points. A \textit{core} of a finite space $X$ is a strong deformation retract of $X$ which is a minimal finite space. The following results are due to Stong \cite{Sto}.

\begin{teo}[Stong] \label{teostong} 
Let $X$ be a minimal finite space. A map $f:X\to X$ is homotopic to the identity if and only if $f=1_X$.
\end{teo}

Every finite space has a core. Moreover we have the following

\begin{coro} [Stong] \label{corostong}
A homotopy equivalence between minimal finite spaces is a homeomorphism. In particular the core of a finite space is unique up to homeomorphism and two finite spaces are homotopy equivalent if and only if they have homeomorphic cores.
\end{coro} 

Note that the core of a finite space $X$ is the smallest space homotopy equivalent to $X$. Thus, $X$ is contractible if and only if $X\scp *$.

From Corollary \ref{corostong} one deduces that two finite $T_0$-spaces are homotopy equivalent if and only if one of them can be obtained from the other just by removing and adding beat points. Thus, the notion of strong homotopy types of finite spaces that would follow from the notion of strong collapse coincides with the usual notion of homotopy types.

These results motivated the definition of strong homotopy types of simplicial complexes and the results proved in Section \ref{sectionsht}. The relationship between both concepts will be given in Theorems \ref{stronght} and \ref{mainstrong}.

We exhibit here a simple generalization of Stong's results for finite topological pairs which uses ideas very similar to the original. As a consequence we prove that strong expansions of finite spaces coincide with strong deformation retracts.

\begin{defi}
Let $(X,A)$ be a pair of finite $T_0$-spaces, i.e. $A$ is a subspace of a finite $T_0$-space $X$. We say that $(X,A)$ is a \textit{minimal pair} if all the beat points of $X$ lie in $A$.
\end{defi}

\begin{prop} \label{stongteopair}
Let $(X,A)$ be a minimal pair and let $f:X\to X$ be a map such that $f\simeq 1_X \textrm{ rel } A$. Then $f=1_X$.
\end{prop}
\begin{proof}
Suppose that $f\le 1_X$ and $f|_A=1_A$. Let $x\in X$. If $x\in X$ is minimal, $f(x)= x$. In general, suppose we have proved that $f|_{\hat{U}_x}=1|_{\hat{U}_x}$. If $x\in A$, $f(x)=x$. If $x\notin A$, $x$ is not a down beat point of $X$. However $y<x$ implies $y=f(y)\le f(x)\le x$. Therefore $f(x)=x$. The case $f\ge 1_X$ is similar, and the general case follows from \ref{homotopias}.
\end{proof}

Note that Theorem \ref{teostong} follows from Proposition \ref{stongteopair} taking $A=\emptyset$.

\begin{coro} \label{minpair}
Let $(X,A)$ and $(Y,B)$ be minimal pairs and let $f:X\to Y$, $g:Y\to X$ be such that $gf\simeq 1_X \textrm{ rel } A$, $gf\simeq 1_Y \textrm{ rel } B$. Then $f$ and $g$ are homeomorphisms.
\end{coro}

\begin{coro} \label{rdf}
Let $X$ be a finite $T_0$-space and let $A\subseteq X$. Then, $X\scp A$ if and only if $A$ is a strong deformation retract of $X$.
\end{coro}
\begin{proof}
We only have to prove that if $A\subseteq X$ is a strong deformation retract, then $X\scp A$. To this end, perform arbitrary elementary strong collapses removing beat points which are not in $A$. Suppose $X\scp Y\supseteq A$ and that all the beat points of $Y$ lie in $A$. Then $(Y, A)$ is a minimal pair. Since $A$ and $Y$ are strong deformation retracts of $X$, the minimal pairs $(A,A)$ and $(Y,A)$ satisfy the hypothesis of Corollary \ref{minpair}. Therefore $A$ and $Y$ are homeomorphic and so, $X\scp Y=A$.
\end{proof}

It follows that $A\subseteq X$ is a core of $X$ if and only if it is a minimal finite space such that $X\scp A$.

\bigskip

The relationship between finite spaces and simplicial complexes is the following. Given a finite simplicial complex $K$, we define the \textit{associated finite space} (\textit{face poset}) $\x (K)$ as the poset of simplices of $K$ ordered by inclusion. McCord proved that $|K|$ and $\x(K)$ have the same weak homotopy type \cite{Mcc}. In particular, they have the same homotopy and homology groups, however in general they are not homotopy equivalent. Conversely, if $X$ is a finite $T_0$-space, the \textit{associated simplicial complex} (\textit{order complex}) $\kp (X)$ has as simplices the non-empty chains of $X$ and it also has the same weak homotopy type as $X$. These constructions are functorial: a simplicial map $\varphi :K\to L$ induces an order preserving map $\x (\varphi) :\x (K) \to \x (L)$ given by $\x (\varphi) (\sigma)=\varphi (\sigma)$ and a continuous map $f:X\to Y$ between finite $T_0$-spaces induces a simplicial map $\kp (f): \kp (X)\to \kp (Y)$ defined by $\kp (f)(x)=f(x)$.

\begin{lema}
Let $f,g: X\to Y$ be two homotopic maps between finite $T_0$-spaces. Then there exists a sequence $f=f_0,f_1, \ldots, f_n=g$ such that for every $0\le i< n$ there is a point $x_i\in X$ with the following properties: 

1. $f_i$ and $f_{i+1}$ coincide in $X\smallsetminus \{x_i\}$, and 

2. $f_i(x_i)$ covers or is covered by $f_{i+1}(x_{i})$. 
\end{lema}
\begin{proof}
Without loss of generality, we may assume that $f=f_0\le g$ by Proposition \ref{homotopias}. Let $A=\{x\in X \ | \ f(x)\neq g(x)\}$. If $A=\emptyset$, $f=g$ and there is nothing to prove. Suppose $A\neq \emptyset$ and let $x=x_0$ be a maximal point of $A$. Let $y\in Y$ be such that $y$ covers $f(x)$ and $y\le g(x)$ and define $f_1:X \to Y$ by $f_1|_{X\smallsetminus \{x\}}=f|_{X\smallsetminus \{x\}}$ and $f_1(x)=y$. Then $f_1$ is continuous for if $x'>x$, $x'\notin A$ and therefore $$f_1(x')=f(x')=g(x')\ge g(x)\ge y=f_1(x).$$

Repeating this construction for $f_i$ and $g$, we define $f_{i+1}$. By finiteness of $X$ and $Y$ this process ends.
\end{proof}

\begin{prop} \label{homotcontig}
Let $f,g: X\to Y$ be two homotopic maps between finite $T_0$-spaces. Then the simplicial maps $\kp(f),\kp(g) :\kp (X)\to \kp(Y)$ lie in the same contiguity class.
\end{prop}
\begin{proof}
By the previous lemma, we can assume that there exists $x\in X$ such that $f(y)=g(y)$ for every $y\neq x$ and $f(x)$ is covered by $g(x)$. Therefore, if $C$ is a chain in $X$, $f(C)\cup g(C)$ is a chain in $Y$. In other words, if $\sigma \in \kp (X)$ is a simplex, $\kp (f)(\sigma)\cup \kp (g)(\sigma)$ is a simplex in $\kp (Y)$.
\end{proof}

\begin{prop} \label{lema1p}
Let $\varphi,\psi:K\to L$ be simplicial maps which lie in the same contiguity class. Then $\x (\varphi)\simeq \x (\psi)$.
\end{prop}
\begin{proof}
Assume that $\varphi$ and $\psi$ are contiguous. Then the map $f :\x(K)\to \x(L)$, defined by $f (\sigma)=\varphi(\sigma)\cup \psi(\sigma)$ is well-defined and continuous. Moreover $\x(\varphi)\le f \ge \x(\psi)$, and then $\x(\varphi)\simeq \x(\psi)$.
\end{proof}

Now we will study the relationship between strong homotopy types of simplicial complexes and homotopy types of finite spaces. The following result is a direct consequence of Propositions \ref{homotcontig} and \ref{lema1p}.

\begin{teo} \label{stronght}  \

\begin{enumerate}
\item[(a)] If two finite $T_0$-spaces are homotopy equivalent, their associated complexes have the same strong homotopy type.
\item[(b)] If two complexes have the same strong homotopy type, the associated finite spaces are homotopy equivalent.
\end{enumerate}
\end{teo}


In fact, we can give a more precise result:

\begin{teo} \label{mainstrong} \

\begin{enumerate} 
\item[(a)] Let $X$ be a finite $T_0$-space and let $Y\subseteq X$. If $X \scp Y$, then $\kp(X) \scp \kp (Y)$.
\item[(b)] Let $K$ be a complex and let $L\subseteq K$. If $K \scp L$, then $\x(K) \scp \x (K)$.
\end{enumerate}
\end{teo}
\begin{proof}
If $x\in X$ is a beat point, there exists a point $x'\in \hat{C}_x$ which is comparable with all the other points of $\hat{C}_x$. Then $lk_{\kp (X)}(x)$ is a simplicial cone. Therefore, $\kp (X)\scp \kp (X) \smallsetminus x= \kp (X\smallsetminus \{x\})$.

If $K$ is a complex and $v\in K$ is such that $lk(v)=aL$ is a simplicial cone, we define $r:\x (K) \to \x (K\smallsetminus v)$ as follows:
\begin{displaymath}
  r(\sigma)=\left\{\begin{array}{ll}
  a\sigma \smallsetminus\{v\} & \textrm{if } v\in \sigma \\
  \sigma & \textrm{if } v\notin \sigma \\
 \end{array} \right.
\end{displaymath}
Clearly $r$ is a well defined order preserving map. Denote $i: \x (K\smallsetminus v) \hookrightarrow \x (K)$ the inclusion and define $f:\x (K) \to \x (K)$, 
\begin{displaymath}
  f(\sigma)=\left\{\begin{array}{ll}
  a\sigma & \textrm{if } v\in \sigma \\
  \sigma & \textrm{if } v\notin \sigma \\
 \end{array} \right.
\end{displaymath}
Then $ir\le f \ge 1_{\x (K)}$ and both $ir$ and $f$ are the identity on $\x (K \smallsetminus v)$. Therefore $ir\simeq 1_{\x(K)} \textrm{ rel } \x (K\smallsetminus v)$ and then $\x (K) \scp \x (K\smallsetminus v)$ by Corollary \ref{rdf}.
\end{proof}


Notice that the barycentric subdivision of a complex $K$ coincides with $K'=\kp (\x (K))$. The \textit{barycentric subdivision} of a finite $T_0$-space $X$ is defined as $X'=\x (\kp (X))$.

\begin{teo} \label{scysubdiv}
Let $K$ be a complex. Then $K$ is strong collapsible if and only if $K'$ is strong collapsible.
\end{teo}
\begin{proof}
If $K\scp *$, then $\x (K)\scp *$ and $K'=\kp (\x (K))\scp *$ by Theorem \ref{mainstrong}. Suppose now that $K$ is such that $K'\scp *$. Let $L$ be a core of $K$. Then $K\scp L$ and by Theorem \ref{mainstrong}, $K' \scp L'$. Therefore $L$ is minimal and $L'$ is strong collapsible. Let $L_0=L', L_1, L_2, ..., L_n=*$ be a sequence of subcomplexes of $L'$ such that there is an elementary strong collapse from $L_i$ to $L_{i+1}$ for every $0\le i <n$. We will prove by induction in $i$ that all the barycenters of the $0$-simplices and of the maximal simplices of $L$ are vertices of $L_i\subseteq L'$.

Let $\sigma =\{v_0, v_1, \ldots , v_k \}$ be a maximal simplex of $L$. By induction, the barycenter $\hat{\sigma}$ of $\sigma$ is a vertex of $L_i$. We claim that $lk _{L_i} (\hat{\sigma})$ is not a cone. If $\sigma$ is a $0$-simplex, that link is empty, so we assume $\sigma$ has positive dimension. Since $\hat{v}_j \hat{\sigma}$ is a simplex of $L'$, $\hat{v}_j\in L_i$ by induction and $L_i$ is a full subcomplex of $L'$, then $\hat{v}_j\in lk _{L_i} (\hat{\sigma})$ for every $0\le j\le k$. Suppose $lk _{L_i} (\hat{\sigma})$ is a cone. In particular, there exists $\sigma ' \in L$ such that $\hat{\sigma} '\in lk _{L_i} (\hat{\sigma})$ and moreover $\hat{\sigma} ' \hat{v}_j\in lk _{L_i} (\hat{\sigma})$ for every $j$. Since $\sigma$ is a maximal simplex, $\sigma ' \subsetneq \sigma$ and $v_j\in \sigma '$ for every $j$. Then $\sigma \subseteq \sigma '$, which is a contradiction. Hence, $\hat{\sigma}$ is not a dominated vertex of $L_i$ and therefore, $\hat{\sigma} \in L_{i+1}$.

Let $v\in L$ be a vertex. By induction, $\hat{v}\in L_i$. As above, if $v$ is a maximal simplex of $L$, $lk _{L_i} (\hat{v})=\emptyset$. Suppose $v$ is not a maximal simplex of $L$. Let $\sigma _0, \sigma _1, \ldots , \sigma _k$ be the maximal simplices of $L$ which contain $v$. By induction $\hat{\sigma} _j\in L_i$ for every $0\le j\le k$, and since $L_i\subseteq L'$ is full, $\hat{\sigma} _j\in lk _{L_i} (\hat{v})$. Suppose that $lk _{L_i} (\hat{v})$ is cone. Then there exists $\sigma \in L$ such that $\hat{\sigma} \in lk _{L_i} (\hat{v})$ and moreover, $\hat{\sigma} \hat{\sigma} _j \in lk _{L_i} (\hat{v})$ for every $j$. In particular, $v\subsetneq \sigma $ and $\sigma \subseteq \sigma _j$ for every $j$. Let $v'\in \sigma$, $v'\neq v$. Then $v'$ is contained in every maximal simplex which contains $v$. This contradicts the minimality of $L$. Therefore $\hat{v}$ is not dominated in $L_i$, which proves that $\hat{v}\in L_{i+1}$.

Finally, $L_n=*$ contains all the barycenters of the vertices of $L$. Thus, $L=*$ and $K$ is strong collapsible.

\end{proof}

We will prove a finite-space version of Theorem \ref{scysubdiv}. In general, $X$ and $X'$ do not have the same homotopy type. However, we will show that $X$ is contractible if and only if its barycentric subdivision $X'$ is contractible. The first implication follows immediately from Theorem \ref{mainstrong}. For the converse, we need some previous results.

%



\begin{lema} \label{compara}
Let $X$ be a finite $T_0$-space. Then $X$ is a minimal finite space if and only if there are no $x,y\in X$ with $x\neq y$ such that if $z\in X$ is comparable with $x$, then so is it with $y$.
\end{lema}
\begin{proof}
If $X$ is not minimal, there exists a beat point $x$. Without loss of generality assume that $x$ is a down beat point. Let $y$ be the maximum of $\hat{U}_x$. Then if $z\ge x$, $z\ge y$ and if $z<x$, $z\le y$.

Conversely, suppose that there exists $x$ and $y$ as in the statement. In particular $x$ is comparable with $y$. We may assume that $x> y$. Let $A=\{z\in X \ | \ z> y $ and $ C_z\subseteq C_y \}$. This set is non-empty since $x\in A$. Let $x'$ be a minimal element of $A$. We show that $x'$ is a down beat point with $y=max(\hat{U}_{x'})$. Let $z<x'$, then $z$ is comparable with $y$ since $x'\in A$. Suppose $z>y$. Let $w\in X$. If $w\le z$, then $w\le x'$ and so, $w\in C_y$. If $w\ge z$, $w\ge y$. Therefore $z\in A$, contradicting the minimality of $x'$. Then $z\le y$. Therefore $y$ is the maximum of $\hat{U}_{x'}$. 
\end{proof}

\begin{prop} \label{minimalyk}
Let $X$ be a finite $T_0$-space. Then $X$ is a minimal finite space if and only if $\kp (X)$ is a minimal complex.
\end{prop}
\begin{proof}
If $X$ is not minimal, it has a beat point $x$ and then $\kp (X)\scp \kp (X\smallsetminus \{ x \} )$ by Theorem \ref{mainstrong}. Therefore $\kp (X)$ is not minimal.

Conversely, suppose $\kp (X)$ is not minimal. Then it has a dominated vertex $x$. Suppose $lk (x)=x'L$ for some $x'\in X$, $L\subseteq \kp (X)$. In particular, if $y\in X$ is comparable with $x$, $y\in lk (x)$ and then $yx'\in lk (x)$. Thus, any point comparable with $x$ is also comparable with $x'$. By Lemma \ref{compara}, $X$ is not minimal.
\end{proof}

\begin{coro} \label{xprima}
Let $X$ be a finite $T_0$-space. Then $X$ is contractible if and only if $X'$ is contractible.
\end{coro}
\begin{proof}
If $X$ is contractible, $X\scp *$, then $\kp (X)\scp *$ by Theorem \ref{mainstrong} and therefore $X'=\x (\kp (X))\scp *$. Now suppose that $X'$ is contractible. Let $Y\subseteq X$ be a core of $X$. Then by Theorem \ref{mainstrong} $X'\scp Y'$. If $X'$ is contractible, so is $Y'$. Again by Theorem \ref{mainstrong}, $\kp (Y')=\kp (Y)'$ is strong collapsible. By Theorem \ref{scysubdiv}, $\kp (Y)$ is strong collapsible. By Proposition \ref{minimalyk}, $\kp (Y)$ is a minimal complex and therefore $\kp (Y)=*$. Then $Y$ is just a point, so $X$ is contractible.
\end{proof}

There is an alternative proof of Corollary \ref{xprima} which does not make use of Theorem \ref{scysubdiv}. It can be obtained from the following result, which is  the converse of Theorem \ref{stronght} (b).

\begin{teo} \label{alter}
Let $K$ and $L$ be two simplicial complexes. If $\x (K)$ and $\x (L)$ are homotopy equivalent, then $K$ and $L$ have the same strong homotopy type.
\end{teo}
\begin{proof}
Let $f:\x (K)\to \x (L)$ be a homotopy equivalence. We may assume that $f$ sends minimal points to minimal points since in other case, it is possible to replace $f$ by a map $\widehat{f}$ which sends a minimal element $x$ of $\x (K)$ to a minimal element $y\le f(x)$ of $\x (L)$. This new map $\widehat{f}$ is also order-preserving and since $\widehat{f}\le f$, it is also a homotopy equivalence.

The map $f$ induces a vertex map $F:K \to L$ defined by $F(v)=w$ if $f(\{v\})=\{w\}$. Moreover $F$ is simplicial for if $\sigma= \{v_0, v_1, \ldots , v_n\}$ is a simplex of $K$, then $\sigma \ge \{v_i\}$ in $\x (K)$ for every $i$ and then $f(\sigma)\ge f(\{v_i\})$ in $\x (L)$. Thus, $F(v_i)$ is contained in the simplex $f(\sigma)$ for every $i$, which means that $F(\sigma)$ is a simplex of $L$.

Let $g: \x (L)\to \x (K)$ be a homotopy inverse of $f$. We may also assume that $g$ sends minimal elements to minimal elements and therefore it induces a simplicial map $G: L\to K$. We show that $GF\sim 1_K$ and $FG\sim 1_L$.

Since $gf\simeq 1_{\x (K)}$, there exists a sequence of maps $h_i:\x (K)\to \x (K)$ such that $$gf= h_0 \le h_1 \ge \ldots \le h_{2k} =1_{\x (K)}.$$

Once again we can assume that the maps $h_{2i}$ send minimal elements to minimal elements and induce simplicial maps $H_{2i}:K \to K$. The maps $H_{2i}$ and $H_{2i+2}$ are contiguous, for if $\sigma=\{v_0, v_1, \ldots , v_n\}$, then $$h_{2i}(\{v_j\})\le h_{2i}(\sigma)\le h_{2i+1}(\sigma),$$ and then $H_{2i}(v_j)\in h_{2i+1}(\sigma)$ for every $j$. Therefore $H_{2i}(\sigma) \subseteq h_{2i+1}(\sigma)$. Analogously $H_{2i+2}(\sigma) \subseteq h_{2i+1}(\sigma)$ and then $H_{2i}(\sigma)\cup H_{2i+2}(\sigma)$ is a simplex of $K$. Hence, $GF=H_0$ lies in the same contiguity class as $H_{2k}=1_K$. Symmetrically $FG\sim 1_L$ and by Corollary \ref{equivstrong}, $K$ and $L$ have the same strong homotopy type.    
\end{proof}

The alternative proof of Corollary \ref{xprima} mentioned above is as follows. Suppose $X$ is a finite $T_0$-space such that $X'$ is contractible. Let $Y$ be a core of $X$. Then $\x (\kp (Y))=Y'\simeq X'$ is also contractible and by Theorem \ref{alter}, $\kp (Y)$ has the strong homotopy type of a point. However, by Proposition \ref{minimalyk}, $\kp (Y)$ is minimal, and then $\kp (Y)$ is a point, which proves that the core $Y$ of $X$ is a point.

\section{From strong collapses to evasiveness}\label{eva}

The theory of strong homotopy types has been shown to be much easier to handle than the classical simple homotopy theory. In order to understand the difference between simplicial collapses and  strong collapses or, more specifically, to investigate how far is a simplicial collapse from being a strong collapse, one can relax the notion of strong collapse and define inductively different notions of collapses which lie between both concepts. In this section we will compare the various notions of collapses of complexes and use the theory of finite topological spaces to improve some known results on collapsibility and non-evasiveness.

The notion of non-evasiveness arises from problems of graph theory and extends then to combinatorics and combinatorial geometry \cite{Bjo3,Kah,Koz,Lut2,Wel}. We will see that this concept appears naturally in our context. Our approach differs slightly from previous treatments of this subject. As we have already mentioned, we are more interested in understanding the difference between the distinct notions of collapses from a geometric viewpoint. As a consequence of our results, we obtain a stronger version of a known result of V. Welker concerning the barycentric subdivisions of collapsible complexes \cite{Wel}.



\begin{defi}
A strong collapse $K\scp L$ will be also called  a  \textit{$0$-collapse} and will be denoted by $K\ceroc L$. A complex $K$ is \textit{$0$-collapsible} if $K$ $0$-collapses to a point (i.e. if it is strong collapsible). We will say that there is a \textit{$1$-collapse} $K\unoc L$ from $K$ to $L$ if we can remove vertices one by one in such a way that their links are $0$-collapsible complexes. A complex is \textit{$1$-collapsible} if it $1$-collapses to a point. In general, there is an \textit{$(n+1)$-collapse} $K \enemasunoc L$ if there exists a sequence of deletions of vertices whose links are $n$-collapsible. A complex is \textit{$(n+1)$-collapsible} if it $(n+1)$-collapses to a point. 
\end{defi}

Since simplicial cones are strong collapsible, if follows that $n$-collapsible complexes are $(n+1)$-collapsible. In fact, this implication is strict for every $n$. Example \ref{nosc} shows a $1$-collapsible complex which is not $0$-collapsible. This example can be easily generalized to higher dimensions.

\begin{ej}
Denote by $S^0$ the discrete complex with two vertices and consider the simplicial join $(S^0)^{*4}=S^0*S^0*S^0*S^0$ with vertices $a,a',b,b',c,c',d,d'$ and whose simplices are the sets of vertices which do not have a repeated letter. Let $K$ be the complex obtained from that space when removing the maximal simplex $\{a,b,c,d\}$ (See Figure \ref{fig:dosco}). The links of the points $a,b,c$ and $d$ are isomorphic to the complex of Example \ref{nosc}. The links of the othe points are spheres isomorphic to $(S^0)^{*3}=S^0*S^0*S^0$. Then $K$ is not $1$-collapsible. However its easy to see that $K$ is $2$-collapsible.
In general, it can be proved that $(S^0)^{*n}$ minus a maximal simplex is $(n-2)$-collapsible but not $(n-3)$-collapsible.
\end{ej}

\begin{figure}[h]
\begin{center}
\includegraphics[scale=1]{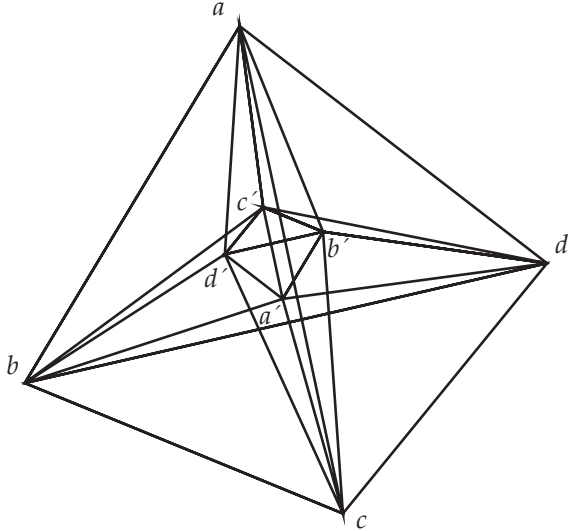}
\caption{A $2$-collapsible simplicial complex which is not $1$-collapsible.}\label{fig:dosco}
\end{center}
\end{figure}

\begin{defi}
A simplicial complex is said to be \textit{non-evasive} if it is $n$-collapsible for some $n\ge 0$. If a complex is not non-evasive, it will be called \textit{evasive}.
\end{defi}

The minimum number $n$ such that $K$ is $n$-collapsible measures how far is the non-evasive complex $K$ from being strong collapsible. Note that this definition of non-evasive coincides with the usual definition, which is as follows: $K$ is non-evasive if it has only one vertex or if inductively there exists a vertex $v\in K$ such that both $K\smallsetminus v$ and $lk (v)$ are non-evasive (cf. \cite{Bjo3,Koz}). 

Note that the notions of non-evasiveness and $n$-collapsibility coincide in complexes of dimension less than or equal to $n+1$.

If $K$ is a complex and $v\in K$ is such that $lk (v)$ is non-evasive, we say that there is an \textit{elementary NE-collapse} from $K$ to $K\smallsetminus v$. We say that there is an \textit{NE-collapse} (or an \textit{NE-reduction}) $K\searrow _{NE} L$ from a complex $K$ to a subcomplex $L$ if there exists a sequence of elementary NE-collapses from $K$ to $L$.
It is easy to prove by induction that non-evasive complexes are collapsible. The converse is not true as the following example shows.

\begin{ej}
The $2$-homogeneous complex of Figure \ref{hachimori} is a slight modification of one defined by Hachimori (see \cite[Section 5.4]{Hac}).
\begin{figure}[h]
\begin{center}
\includegraphics[scale=1]{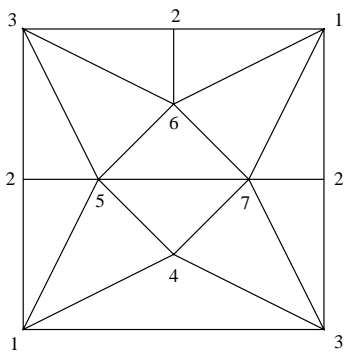}
\caption{A collapsible evasive complex.}\label{hachimori}
\end{center}
\end{figure}
The face $\{1,3\}$ is free and it is easy to see that the complex is collapsible. However, the links of the vertices are graphs which are not trees. Thus it is evasive. 
\end{ej}

\begin{defi}
A finite $T_0$-space $X$ is \textit{non-evasive} if it is a point or inductively if there exists $x\in X$ such that both its link $\hat{C}_x$ and $X\smallsetminus \{x\}$ are non-evasive.

If $x$ is a point of a finite $T_0$-space $X$ such that $\hat{C}_x$ is non-evasive, we say that there is an \textit{elementary NE-collapse} from $X$ to $X\smallsetminus \{x\}$. There is an \textit{NE-collapse} from $X$ to a subspace $Y$ if there exists a sequence of elementary NE-collapses from $X$ to $Y$. We denote this by $X\searrow _{NE} Y$.
\end{defi}

Recall from \cite{Bar2} and \cite{Bar3} that a \textit{weak point} of a finite $T_0$-space $X$ is a point $x\in X$ such that $\hat{C}_x$ is contractible. We say that there is a \textit{collapse} $X\searrow Y$ if the subspace $Y$ can be obtained from $X$ by removing weak points one by one. Beat points are particular cases of weak points.

\begin{obs}
If $X$ is a finite $T_0$-space such that there is a point $x\in X$ with $C_x=X$, it can be proved by induction that $X$ is non-evasive. The link of a beat point satisfies this property and therefore, strong collapses of finite spaces are NE-collapses and contractible spaces are non-evasive. Since the links of weak points are contractible, collapses are NE-collapses and collapsible spaces are non-evasive. 
\end{obs}

Although it seems awkward, for complexes we have that a NE-collapse $K\searrow _{NE} L$ is a collapse $K\searrow L$, but for finite spaces the implication is in the opposite direction: a collapse $X\searrow Y$ is a NE-collapse $X\searrow _{NE} Y$. However, the following results will show that the notion of NE-collapse for finite spaces corresponds exactly to the notion of NE-collapse of simplicial complexes. These results extend our previous results of \cite{Bar2} on simple homotopy types.

\begin{teo} \label{mainne}
Let $X$ and $Y$ be finite $T_0$-spaces. If $X\searrow _{NE} Y$, then $\kp (X) \searrow _{NE} \kp (Y)$.
\end{teo}
\begin{proof}
We prove the result by induction in the number of points of $X$. It suffices to corroborate it for elementary NE-collapses. Let $x\in X$ be such that $\hat{C}_x$ is non-evasive. Then $\hat{C}_x \searrow _{NE} *$ and by induction $lk _{\kp (X)} (x)=\kp (\hat{C}_x)\searrow _{NE} *$, i.e. $lk _{\kp (X)} (x)$ is non-evasive. Therefore $\kp (X)\searrow _{NE} \kp(X) \smallsetminus x=\kp (X\smallsetminus \{x\})$.  
\end{proof}

In particular one deduces the analogous version for (usual) collapses proved in \cite{Bar2}. For the sake of completeness we include the proof of the following result from \cite{Bar2}.

\begin{teo} \label{viejo}
If a complex $K$ collapses to a subcomplex $L$, then $\x (K)\searrow \x (L)$.
\end{teo}
\begin{proof}
We may assume there is an elementary collapse $K\searrow L$ of a free face $\sigma$ of a simplex $a\sigma$. Then $\sigma \in \x (K)$ is covered only by $a\sigma$. Therefore $\sigma$ is a beat point and in particular $\x (K) \searrow \x (K)\smallsetminus \{\sigma\}$. Now, the link of $a\sigma$ in $\x (K)\smallsetminus \{\sigma\}$ is $\x (a \dot{\sigma})$ which is contractible by Theorem $\ref{mainstrong}$ since the cone $a\dot{\sigma}$ is strong collapsible. Then, $a\sigma$ is a weak point of $\x (K)\smallsetminus \{\sigma\}$ which proves that $\x (K)\smallsetminus \{\sigma\}\searrow \x (L)$.  
\end{proof}

In particular we deduce the following

\begin{coro} \label{mainnep}
If $K\searrow _{NE} L$ is a NE-collapse of complexes, then $\x(K) \searrow _{NE} \x (L)$. 
\end{coro}

\begin{teo} \label{nueva}
Let $Y$ be a subspace of a finite $T_0$-space $X$. Then, $X\searrow Y$ if and only if $\kp (X)\unoc \kp (Y)$. In particular $X$ is collapsible if and only if $\kp (X)$ is $1$-collapsible.
\end{teo}
\begin{proof}
Since $\kp (X)\smallsetminus x =\kp (X\smallsetminus \{x\})$, it suffices to analize the case that $Y=X\smallsetminus \{x\}$ for some $x\in X$. There is a collapse from $X$ to $X\smallsetminus \{x\}$ if and only if $x\in X$ is a weak point, or in other words if $\hat{C}_x$ is contractible. But if $\hat{C}_x$ is contractible, then by Theorem \ref{mainstrong}, $lk_{\kp (X)}(x)=\kp (\hat{C}_x)$ is strong collapsible. Conversely, if $lk_{\kp (X)}(x)=\kp (\hat{C}_x)$ is strong collapsible, $(\hat{C}_x)'$ is contractible by Theorem \ref{mainstrong} and then $x\in X$ is a weak point by Theorem \ref{xprima}. Finally, the link of $x$ in $\kp (X)$ is strong collapsible if and only if $\kp (X)\unoc \kp (X\smallsetminus \{x\})$. 
\end{proof}

We can now derive the most important result of this section:

\begin{coro}
If a complex $K$ collapses to a subcomplex $L$, then $K'\unoc L'$. In particular if $K$ is collapsible, $K'$ is $1$-collapsible.
\end{coro}
\begin{proof}
If $K\searrow L$, then by Theorem \ref{viejo} $\x (K)\searrow \x (L)$ and by Theorem \ref{nueva}, $K'=\kp( \x (K))\unoc \kp( \x (L))=L'$.
\end{proof}

In view of this result, it is equivalent to say that a compact polyhedron has a collapsible triangulation and that it has a $1$-collapsible triangulation. This equivalence could be helpful when studying problems of collapsibility, such as Zeeman's conjecture.

In particular one can deduce Welker's result  \cite[Thm. 2.10]{Wel}.

\begin{coro} [Welker] \label{welker}
Let $K$ be a collapsible complex. Then the barycentric subdivision $K'$ is non-evasive.
\end{coro}


\begin{obs}
It is unknown whether the converse of Corollary \ref{welker} holds. However if $K'$ is not just non-evasive but strong collapsible, then by Theorem \ref{scysubdiv}, $K$ is strong collapsible and in particular collapsible.

The proof of the fact that if $K$ is collapsible then $K'$ is $1$-collapsible was made in two steps: $K\searrow * \Rightarrow \x (K) \searrow * \Rightarrow K'=\kp (\x (K))\unoc *$. In order to prove the converse, one could try to prove the converses of both implications. In fact, the second implication is an equivalence $\x (K) \searrow * \Leftrightarrow K'\unoc *$, by Theorem \ref{nueva}. 
\end{obs}

Of course, there is an analogous of Corollary \ref{welker} for finite spaces. If $X$ is a non-evasive finite space, $\kp (X)$ is non-evasive and then $X'$ is collapsible. If the converse of this statement is true, then any barycentric subdivision $K'$ which is collapsible, would also be non-evasive, and therefore the converse of \ref{welker} would be false.
 
\section{Vertex-homogeneous complexes}

A complex $K$ is said to be \textit{vertex-homogeneous} if the group of simplicial automorphisms $Aut(K)$ acts transitively on the vertices of $K$. In other words, $K$ is vertex-homogeneous if for any two vertices $v,w\in K$, there exists an automorphism $\varphi :K\to K$ such that $\varphi (v)=w$. For instance, the simplices are vertex-homogeneous simplicial complexes. There exist examples of contractible (even collapsible) vertex-homogeneous complexes which are not simplices (see \cite{Kah,Lut1}). The Evasiveness conjecture for simplicial complexes states that a non-evasive vertex-homogeneous simplicial complex is a simplex. This has been verified for complexes with a prime-power number of vertices, but the general question remains open.

In this section we will prove that the conjecture holds if the complex is strong collapsible (Corollary \ref{necsc}). We will also prove that  if $K$ is vertex-homogeneous, its core is vertex-homogeneous and it is isomorphic to $\nerve ^2 (K)$. Moreover, we will see that vertex-homogeneous complexes are isomorphic to an $n$-th multiple of their cores. From these results we will deduce that any vertex-homogeneous complex with a prime number of vertices is minimal or a simplex. Finally, we use the notions of collapses introduced in the previous section to prove that the core of a vertex-homogeneous non-evasive complex is also non-evasive. This result implies that the study of the Evasiveness conjecture can be reduced to the class of minimal complexes.

If $\varphi :K \to K$ is a simplicial automorphism of a contractible complex $K$, there is a point of $|K|$ which is fixed by $|\varphi|$. In fact, every continuous map $f:|K|\to |K|$ has a fixed point by the Lefschetz fixed-point theorem \cite{Spa}. The problem becomes harder when one studies fixed points of an automorhism group instead of a single automorphism. There exist examples of fixed point free actions of finite groups over compact contractible polyhedra \cite{Flo}. Furthermore, R. Oliver \cite{Oli} characterized the groups $G$ for which every simplicial action of $G$ over a finite contractible complex has a fixed point.

From a different perspective, we use a result of Stong \cite{Sto2}, to prove that if $K$ is a strong collapsible complex, then every simplicial action on $K$ has a fixed point.


\begin{teo}[Stong] \label{fixedstong}
If $X$ is a contractible finite $T_0$-space, there is a point $x\in X$ fixed by every homeomorphism of $X$.
\end{teo}

\begin{teo} \label{equivariantlefschetz}
Let $K$ be a strong collapsible complex and let $G$ be a group of simplicial automorphisms of $K$. Then there exists a point $x\in |K|$ such that $|\varphi|(x)=x$ for every $\varphi \in G$.
\end{teo}
\begin{proof}
The action of $G$ on $K$ induces an action of $G$ on the contractible space $\x(K)$. By Theorem \ref{fixedstong} there exists a point $\sigma \in \x (K)$ which is fixed by the action of $G$. Therefore, the simplex $\sigma$ is fixed by the automorphisms of $G$ and then, the barycenter of $\sigma$ is a fixed point by the action.
\end{proof}

\begin{obs}
The fact of being the group of automorphisms of a strong collapsible complex does not impose any restriction on the group. In other words, every finite group is the automorphism group of some strong collapsible complex. Given a finite group $G$, there exists a complex $K$ such that $Aut(K)\simeq G$ (see \cite{Fru}). If $K$ is not a cone, then the simplicial cone $aK$ with base $K$ is a strong collapsible complex with automorphism group isomorphic to $G$.  
\end{obs}

From Theorem \ref{equivariantlefschetz}, we deduce that the Evasiveness conjecture is true for strong collapsible simplicial complexes.

\begin{coro} \label{necsc}
Let $K$ be a strong collapsible vertex-homogeneous complex. Then $K$ is a simplex.
\end{coro}
\begin{proof}
By Theorem \ref{equivariantlefschetz} there exists a point $x\in |K|$ fixed by $Aut(K)$. Therefore, the support of $x$ is a fixed simplex, and since $K$ is vertex-homogeneous, this simplex must contain all the vertices of $K$. 
\end{proof}

We will see another proof of this result as an application of the study of the square-nerve $\nerve ^2(K)$ of a vertex-homogeneous complex $K$.

Let $K$ be a complex and $v,w$ vertices of $K$. We denote $v\prec w$ if every maximal simplex of $K$ which contains $v$, also contains $w$. In other words, $v\prec w$ if $v$ is dominated by $w$ or if $v=w$. Clearly $\prec$ is a preorder on the set of vertices of $K$.

\begin{obs} \label{morfismo}
Let $v,w$ be vertices of a complex $K$ and let $\varphi \in Aut(K)$. If $v\prec w$, $\varphi (v)\prec \varphi (w)$. This follows immediately from the fact that $\varphi$ induces a correspondence between maximal simplices containing $v$ and maximal simplices containing $\varphi(v)$.
\end{obs}

\begin{prop}
If $K$ is vertex-homogeneous, $\prec$ is an equivalence relation and all the equivalence classes have the same cardinality.
\end{prop}
\begin{proof}
Suppose $v\prec w$. Let $\varphi \in Aut (K)$ such that $\varphi (v)=w$. By Remark \ref{morfismo}, $$v\prec \varphi (v) \prec \varphi ^2 (v) \prec \ldots .$$
Since $K$ is finite and $\varphi$ is one-to-one on vertices, there exists $n\ge 1$ such that $\varphi ^n (v)=v$. Then, $w=\varphi (v)\prec \varphi ^n (v)=v$. Thus, $\prec$ is symmetric and then an equivalence.

For the second part, suppose $[v]$ is the equivalence class of a vertex $v$ and $[v']$ is the class of another vertex $v'$. Let $\psi$ be an automorphism such that $\psi (v)=v'$. By Remark \ref{morfismo}, $\psi$ induces a well defined map from $[v]$ to $[v']$ and $\psi ^{-1}$ a map $[v']\mapsto [v]$. These are mutually inverse, and then $[v]$ and $[v']$ have the same cardinality.  
\end{proof}

\begin{obs} \label{m2vh}
If $K$ is vertex-homogeneous, $\nerve ^2 (K)$ is isomorphic to any full subcomplex of $K$ spanned by a set of vertices constituted by one element of each class of $\prec$.
By the proof of Proposition \ref{aprop11},  $\nerve ^2 (K)$ is isomorphic to any full subcomplex of $K$ spanned by a set of vertices, one $v\in \bigcap\limits_{i=0}^{r}\sigma _i$ for each maximal intersecting family $\{\sigma_0, \sigma_1, \ldots , \sigma_r\}$ of maximal simplices of $K$. On the other hand, each one of the intersections $\bigcap\limits_{i=0}^{r}\sigma _i$ is an equivalence class of $\prec$ and viceversa. 
\end{obs}

\begin{teo} \label{m2eselcore}
Let $K$ be a vertex-homogeneous complex. Then, the core of $K$ is vertex-homogeneous and it is isomorphic to $\nerve ^2 (K)$. 
\end{teo}
\begin{proof}
Let $L$ be a full subcomplex of $K$ with one vertex in each equivalence class of $\prec$. Since $K \scp L$ and $L$ is isomorphic to $\nerve ^2 (K)$, we only need to prove that $L$ is minimal. First we claim that if $\sigma$ is a maximal simplex of $K$, $\sigma \cap L$ is a maximal simplex of $L$.

Let $\tau$ be a simplex of $L$ containing $\sigma \cap L$ and let $\nu$ be a maximal simplex of $K$ containing $\tau$. Let $v\in \sigma$. Then there exists $v'\in L$ with $v\prec v'$. Since $v\in \sigma$, then $v'\in \sigma$ and therefore $v'\in \sigma \cap L \subseteq \nu$. Since $\nu$ is maximal in $K$ and $v'\prec v$, $v\in \nu$. This proves that $\sigma \subseteq \nu$ and by the maximality of $\sigma$, $\sigma=\nu$. Thus, $\tau \subseteq \sigma$ and then $\tau =\tau \cap L \subseteq \sigma \cap L$. This proves that $\sigma  \cap L$ is maximal in $L$.

Now, suppose that $v$ and $w$ are two vertices of $L$ and that $v \prec w$ in $L$. Let $\sigma$ be a maximal simplex of $K$ which contains $v$. Then $\sigma \cap L$ is maximal in $L$ and therefore $w\in \sigma \cap L \subseteq \sigma$. Therefore $v\prec w$ in $K$ and since both vertices lie in $L$, $v=w$. Hence, $L$ has no dominated vertices, which says that it is minimal.

It only remains to prove the vertex-homogeneity of $\nerve ^2 (K)$. An automorphism $\varphi :K\to K$ induces an automorphism $\overline{\varphi}:\nerve ^2 (K)\to \nerve ^2 (K)$. If $\{\sigma_0, \sigma_1, \ldots , \sigma_r\}$ is a maximal intersecting family of maximal simplices of $K$, so is $\{\varphi(\sigma_0), \varphi(\sigma_1), \ldots , \varphi(\sigma_r)\}$. This defines a vertex map $\overline{\varphi}: \nerve ^2 (K) \to \nerve ^2 (K)$ which is clearly simplicial with inverse $\overline{\varphi ^{-1}}$.

Suppose $\Sigma _1=\{\sigma_0, \sigma_1, \ldots , \sigma_r\}$ and $\Sigma _2=\{\tau_0, \tau_1, \ldots , \tau_s\}$ are two maximal intersecting families of maximal simplices of $K$. Take $v\in  \bigcap\limits_{i=0}^{r}\sigma _i$, $w\in  \bigcap\limits_{i=0}^{s}\tau _i$ and an automorphism $\varphi$ of $K$ such that $\varphi (v)=w$. Then $\overline{\varphi} (\Sigma _1)=\Sigma _2$, so $\nerve ^2 (K)$ is vertex-homogeneous.
\end{proof}

The fact that the square nerve preserves vertex-homogeneity can also be deduced immediately from \cite[Lemma 10]{Lut1}. 

We adopt Lutz's terminology \cite{Lut1} and call the \textit{$n$-th multiple} of a complex $K$ the complex $nK$ whose vertex set is $V_K\times \{1,2, \ldots , n\}$ and whose simplices are the sets whose projections to $V_K$ are simplices of $K$. In other words, $nK$ is the simplicial product of $K$ with an $(n-1)$-simplex. Notice that if $K$ is vertex-homogeneous, so is $nK$ for every $n\ge 1$.

\begin{prop} \label{multi}
Let $K$ be a vertex-homogeneous complex. Then, there exists $n\ge 1$ such that $K$ is isomorphic to the $n$-th multiple of its core.
\end{prop}
\begin{proof}
Let $L$ be a full subcomplex of $K$ containing one vertex of each equivalence class of $\prec$ and let $n$ be the number of elements of each class. For every $v\in L$, let $\varphi_v:[v]\to \{v\}\times \{1,2, \ldots, n \}$ be a bijection. We claim that the vertex bijective map $\varphi : V_K\to V_{nL}$ induced by the maps $\varphi _v$ is a simplicial isomorphism between $K$ and $nL$.

Let $p:nL\to L$ be the canonical projection. Let $\sigma \in K$. If $\tau$ is a maximal simplex of $K$ which contains $\sigma$ and $v\in p(\varphi(\sigma))$, there exists $v'\in [v]$ such that $v'\in \sigma$. Since $v'\prec v$, $v\in \tau$. Therefore, $p(\varphi(\sigma))\subseteq \tau$ is a simplex, which proves that $\varphi$ is simplicial.

Now if $\sigma \in nL$, $p(\sigma)\in K$. Let $\tau$ be a maximal simplex of $K$ which contains $p(\sigma)$. If $v\in \varphi ^{-1}(\sigma)$, there exists $v'\in [v]$ such that $v'\in p(\sigma)$. Since $v'\prec v$, $v\in \tau$. Thus, $\varphi ^{-1}(\sigma) \subseteq \tau$ is a simplex of $K$ and so, $\varphi ^{-1}$ is simplicial.

Since $L$ is isomorphic to the core of $K$ by Theorem \ref{m2eselcore} and Remark \ref{m2vh}, the result is proved.  
\end{proof}

From this result we obtain an alternative proof of Corollary \ref{necsc}. If $K$ is vertex-homogeneous and strong collapsible, the core $\nerve ^2 (K)$ is a point. The $n$-th multiple of a point is a simplex.

\begin{coro}
If a vertex-homogeneous complex $K$ has a prime number of vertices, then it is minimal or it is a simplex. 
\end{coro}

\begin{defi}
Let $K$ be a complex with vertex set $V_K$ and let $\mathcal{F}: V_K \to \mathbb{N}$ be a function. We define the complex $\mathcal{F}K$ whose vertices are the pairs $(v,i)$ with $v\in V_K$ and $1\le i\le \mathcal{F}(v)$ and whose simplices are the sets $\{(v_0, i_0), (v_1, i_1), \ldots , (v_r,i_r)\}$ such that the projections $\{v_0,v_1, \ldots ,v_r\}$ are simplices of $K$.

This generalizes the notion of $n$-multiple of $K$ which coincides with $\mathcal{F}K$ when $\mathcal{F}$ is the constant map $n$.
\end{defi}

\begin{teo} \label{reduccion}
Let $K$ be a simplicial complex, $\mathcal{F}:V_K\to \mathbb{N}$ and $n\ge 0$. If $\mathcal{F}K$ is $n$-collapsible, then $K$ is $n$-collapsible. In particular if $\mathcal{F}K$ is non-evasive, so is $K$. 
\end{teo}
\begin{proof}
The proof is by induction in $n$. The full subcomplex $L$ of $\mathcal{F}K$ spanned by the vertices of the form $(v,1)$ with $v\in V_K$ is isomorphic to $K$. Moreover, every vertex $(v,i)$ of $\mathcal{F}K$ with $i\ge 2$ is dominated by the vertex $(v,1)$ of $L$. By Lemma \ref{alema1}, $\mathcal{F}K \scp L$. Thus, if $\mathcal{F}K$ is strong collapsible, so is $K$. This proves the case $n=0$.

Now assume that $\mathcal{F}K$ is $(n+1)$-collapsible. Then there exists an ordering $$(v_0,i_0), (v_1,i_1), \ldots (v_r,i_r)$$ of the vertices of $\mathcal{F}K$ such that $lk _{L_s} ((v_s, i_s))$ is $n$-collapsible for $0\le s< r$, where $L_0=\mathcal{F}K$ and $L_{s+1}=L_s \smallsetminus (v_s, i_s)$.

For each $v\in K$ consider the number $l_v=max \{j \ | \ v_j=v\}$. Let $w_0, w_1, \ldots , w_k$ be the ordering of the vertices of $K$ such that $l_{w_0}< l_{w_1} < \ldots < l_{w_k}$. We claim that it is possible to make the deletions of the vertices of $K$ in that order to show that $K\enemasunoc w_k$.

Define $K_0=K$ and $K_{q+1}=K_q \smallsetminus w_q$ for $0\le q <k$. We have to show that $lk _{K_q} (w_q)$ is $n$-collapsible for every $0\le q< k$. We already know that $lk _{L_{l_{w_q}}} ((w_q, i_{l_{w_q}}))$ is $n$-collapsible, so, by induction it suffices to prove that $lk _{L_{l_{w_q}}} ((w_q, i_{l_{w_q}}))$ is isomorphic to $\mathcal{F}'lk _{K_q} (w_q)$ for some $\mathcal{F}': V_{lk _{K_q} (w_q)}\to \mathbb{N}$.

The vertices of $lk _{L_{l_{w_q}}} ((w_q, i_{l_{w_q}}))$ are the vertices $(u,i)$ of $L_{l_{w_q}+1}$ such that $\{(u,i),$ $(w_q, i_{l_{w_q}})\}$ is a simplex of $\mathcal{F}K$. By definition of $l_{w_q}$, if $(u,i)\in L_{l_{w_q}+1}$, then $u=w_p$ for some $p>q$. Thus if $(u,i)\in lk _{L_{l_{w_q}}} ((w_q, i_{l_{w_q}}))$, $u\in lk _{K_q} (w_q)$. Conversely, if $u\in lk _{K_q} (w_q)$, then $u=w_p$ for some $p>q$ and then $(u,i_{l_u}) \in lk _{L_{l_{w_q}}} ((w_q, i_{l_{w_q}}))$. Moreover $(u,i_s) \in lk _{L_{l_{w_q}}} ((w_q, i_{l_{w_q}}))$ for every $l_{w_q} < s \le r$ such that $v_s=u$.    .

In general, a set of vertices $\{(u_0,i_0),(u_1,i_1), \ldots , (u_t,i_t)\}$ of $lk _{L_{l_{w_q}}} ((w_q, i_{l_{w_q}}))$ is a simplex if and only if $\{(u_0,i_0),(u_1,i_1), \ldots , (u_t,i_t), (w_q, i_{l_{w_q}})\}$ is a simplex of $\mathcal{F}K$ or equivalently if $\{u_0, u_1, \ldots, u_t\}\in lk _{K_q} (w_q)$. Then $lk _{L_{l_{w_q}}} ((w_q, i_{l_{w_q}}))$ is isomorphic to $\mathcal{F}'lk _{K_q} (w_q)$ where $\mathcal{F}'(u)$ is the cardinality of the set $\{ s \ | \ l_{w_q} < s \le r$ and $v_s=u \}$.    
\end{proof}

\begin{coro} \label{qwe}
The core of a vertex-homogeneous non-evasive complex is non-evasive.
\end{coro}
\begin{proof}
Let $K$ be a vertex-homogeneous complex and let $L$ be its core. By Proposition \ref{multi}, $K=nL$ fore some $n$. If $K$ is non-evasive, then $L$ is non-evasive by Theorem \ref{reduccion}.
\end{proof}

In general, the core of a non-evasive complex need not be non-evasive as the following example shows.

\begin{ej}
The finite space $X$ whose Hasse diagram is as shown in Figure \ref{monstruom}, is collapsible and therefore $\kp (X)$ is a $1$-collapsible complex. In particular $\kp (X)$ is non-evasive. The point $x\in X$ is a beat point and then $\kp (X)\scp \kp (X\smallsetminus \{x\})$. The space $X\smallsetminus \{x\}$ has no weak points (cf. \cite{Bar3}) and therefore $\kp (X\smallsetminus \{x\})$ has no vertices with strong collapsible links by Theorem \ref{nueva}. In particular $\kp (X\smallsetminus \{x\})$ is the core of $\kp (X)$ and it is evasive because it is two-dimensional and not $1$-collapsible.

\begin{figure}[h!]
\begin{center}
\includegraphics[scale=0.8]{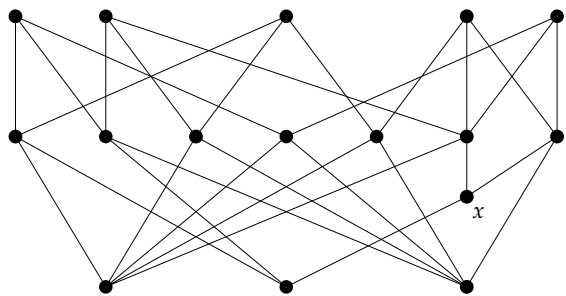}
\caption{A collapsible finite space with a non-collapsible core.}\label{monstruom}
\end{center}
\end{figure}
\end{ej}

In view of Corollary \ref{qwe} and Proposition \ref{multi}, in order to prove the Evasiveness conjecture for simplicial complexes it suffices to verify it for minimal complexes. In other words, the following statement is equivalent to the Evasiveness conjecture. 

\begin{conj}[Evasiveness conjecture for minimal complexes]
If $K$ is a minimal, vertex-homogeneous, non-evasive complex, then it is a point. 
\end{conj}

\end{document}